\documentclass[12pt,twoside,a4paper,reqno]{amsart}
\usepackage{amsmath,amsthm,amssymb,amscd,amsfonts}
\usepackage{enumerate}
\makeatletter
\@namedef{subjclassname@2010}{%
  \textup{2010} Mathematics Subject Classification}
\makeatother
\newtheorem{theorem}{Theorem}[section]

\theoremstyle{definition}

\newtheorem{example}[theorem]{Example}

\newcommand{\cA}{\mathcal A}

\newtheorem{theo}{Theorem}[section]

\newtheorem{lemm}[theo]{Lemma}
\newtheorem{prop}[theo]{Proposition}
\theoremstyle{definition}
\newtheorem{defi}[theo]{Definition}
\theoremstyle{definition}

\newtheorem{rem}[theo]{Remark}
\numberwithin{equation}{section}

\title{Biflat-like Banach algebras  }

\author{Sanaz Haddad sabzevar}
\address{Department of Mathematics, Central Tehran Branch, Islamic Azad University,Tehran, Iran, e-mail: {\tt san.haddadsabzevar.sci@iauctb.ac.ir}}

\author{Amin Mahmoodi}
\address{Department of Mathematics, Central Tehran Branch, Islamic Azad University, Tehran, Iran, e-mail: {\tt a\_mahmoodi@iauctb.ac.ir}}

\begin{document}
\pagestyle{headings}

\begin{abstract}
Given a Banach algebra $ \mathcal{A} $ and a continuous homomorphism
$\sigma$ on it, the notion of $\sigma$-biflatness  for $ \mathcal{A}
$ is introduced. This is a generalization of biflatness and it is
shown that they are distinct. The relations between
$\sigma$-biflatness and some other close concepts such as
$\sigma$-biprojectivity and $\sigma$-amenability are studied. The
$\sigma$-biflatness of tensor product of Banach algebras are also
discussed.
\end{abstract}

\maketitle

\section{Introduction }
Biprojectivity and biflatness for Banach algebras, as introduced and
studied in the works of Helemskii (see for instance [5]), have
proved to be important and fertile notions. There are close
relationship between these notions and some other concepts of Banach
algebras such as amenability. It is known that every biprojective
Banach algebra with a bounded approximate identity is amenable and
in the presence of a bounded approximate identity, biflatness and
amenability are the same notions \cite{H}.

Before preceding further, we recall some preliminaries. Let  $
\mathcal{A} $ be a Banach algebra. Then its projective tensor
product $ \mathcal{A}\widehat{\otimes}{\mathcal{A}}$ is a Banach $
\mathcal{A} $-bimodule through $$  a\cdot(b\otimes c)=ab\otimes c \
\ \ \text{and} \ \ \
  (b\otimes c)\cdot a=b\otimes ca  \ \ \ (a,b,c\in \mathcal{A}).$$
For a Banach $ \mathcal{A} $-bimodule $X$, the dual space $X^*$
becomes a Banach $ \mathcal{A} $-bimodule in a natural manner. Let
$X$ and $Y$ be Banach $ \mathcal{A} $-bimodules. A bounded linear
map $T : X \longrightarrow Y$ is an $ \mathcal{A} $-\textit{bimodule
homomorphism} if $ T(a\cdot x)=a \cdot T(x) $ and $ T(x\cdot a)=
T(x)\cdot a $ , for $a \in \cA$, $x \in X$. It is obvious that the
\textit{diagonal operator} $
\pi_{\mathcal{A}}:\mathcal{A}\widehat{\otimes}\mathcal{A}
\rightarrow \mathcal{A}$ given by  $ \pi( a\otimes b)  = ab  $, is
an $ \mathcal{A} $-bimodule homomorphism. If it is clear to which
algebra $ \mathcal{A} $ we refer, we simply write $\pi$.

A Banach algebra $ \mathcal{A} $ is \textit{biprojective} if $ \pi$
has a right inverse which is an $\mathcal{A}$-bimodule homomorphism.
If there is an $\mathcal{A}$-bimodule homomorphism which is a left
inverse of $\pi^*$, then we say that $ \mathcal{A} $ is
\textit{biflat}.

Let $\mathcal{A}$ be a Banach algebra. We write $Hom(\mathcal{A})$
for the set of all continuous homomorphisms on $ \mathcal{A} $. Let
$X$ and $Y$ be Banach $\mathcal{A}$-bimodule, and let $\sigma\in
Hom(\mathcal{A})$. A bounded linear map $ T:X \rightarrow Y $ is a
$\sigma$-$ \mathcal{A} $-\textit{bimodule homomorphism} if $
T(a\cdot x)=\sigma (a) \cdot T(x) $ and $ T(x\cdot a)=
T(x)\cdot\sigma (a) $ where $ x\in X$ and $a\in\mathcal{A} $. A
Banach algebra $ \mathcal{A} $ is $\sigma$-\textit{biprojective} if
there exists a $ \sigma $-$ \mathcal{A} $-bimodule homomorphism $
\rho : \mathcal{A} \longrightarrow
\mathcal{A}\widehat{\otimes}\mathcal{A}$ such that $ \pi \circ \rho
= \sigma$ [13].

The purpose of this paper is to study the concept of
$\sigma$-biflatness for Banach algebras. We have to stress that our
definition is completely different from what have introduced in [4,
Definition 2.11]. For comparison, unlike our definition,
$\sigma$-biflatness in [4] is not a generalization of the notion of
biflatness [4, Remark 2.13].

The organization of the paper is as follows. Firstly, in section 2
we investigate some basic properties of $\sigma$-biflat Banach
algebras. We find an equivalent condition to $\sigma$-biflatness
(Theorem 2.3). We prove that every $\sigma$-biprojective Banach
algebra is $\sigma$-biflat (Proposition 2.4). However biflat Banach
algebras are $\sigma$-biflat (Proposition 2.5), we give an example
to show that the class of $\sigma$-biflat Banach algebras is larger
than that for biflat Banach algebras (Example 2.1).

In section 3, we find the relations between $\sigma$-biflatness and
both $\sigma$-amenability and the existence of some certain
$\sigma$-diagonals. There are examples of $\sigma$-biflat Banach
algebras which are not $\sigma$-amenable (Examples 3.1 and 3.2).

In section 4, we deal with the short exact sequence $ \Sigma : 0
\longrightarrow ker \pi \stackrel \imath \longrightarrow \mathcal{A}
\widehat{\otimes}\mathcal{A} \stackrel \pi \longrightarrow
\mathcal{A} \longrightarrow 0 $, where $ \mathcal{A}$ is a Banach
algebra. We prove that if $ \mathcal{A}$ is $\sigma$-amenable, then
$ \Sigma ^*$ and $ \Sigma ^{**}$ behave like splitting sequences
(Proposition 4.2).

Finally in section 5, we generalize Ramsden's theorem [11,
Proposition 2.5] related to biflatness of tensor product of Banach
algebras to the $\sigma$-case (Theorem 5.1).


\section{Basic properties}
\begin{defi} \label{2.3}
Let $ \mathcal{A} $ be a Banach algebra and let $ \sigma\in
Hom(\mathcal{A}) $. Then $ \mathcal{A} $ is \textit{$ \sigma
$-biflat} if there exists a bounded linear map $ \rho : (
\mathcal{A}\widehat{\otimes}\mathcal{A})^* \longrightarrow
\mathcal{A}^* $ satisfying
$$\rho(\sigma(a)\cdot\lambda)=a\cdot\rho(\lambda) \ \ \ \text{and}
\ \ \ \rho(\lambda\cdot\sigma(a))=\rho(\lambda)\cdot a \ \ \ \ \ \ \
\ \ \ \ \ \ \ \ \ (1)$$ for $ a\in \mathcal{A} ,
\lambda\in\mathcal{A}^{\ast}$ such that $ \rho \circ \pi^* =
\sigma^*$.
\end{defi}
\begin{lemm}\label{2.2}
Let $ \mathcal{A} $ be a Banach algebra, let $ \sigma\in
Hom(\mathcal{A}) $ and let $ X $ and $ Y $ be Banach $ \mathcal{A}
$-bimodules. If $ T:X\rightarrow Y $ is a $ \sigma $-$ \mathcal{A}
$-bimodule homomorphism, then $ T^{\ast}$ satisfies $(1) $,
 that is, $ T^{\ast}(\sigma(a)\cdot\lambda)=a\cdot T^{\ast}(\lambda) $ and
  $T^{\ast}(\lambda\cdot\sigma(a))=T^{\ast}(\lambda)\cdot a$, for $a\in\mathcal{A},\lambda\in Y^{\ast}$.
\end{lemm}

\begin{proof}
For every $ x\in X $, $ a\in A $ and $ \lambda\in Y^{\ast} $, we have
\begin{align*}
\langle T^{\ast}(\sigma(a)\cdot\lambda),x\rangle &=\langle\sigma(a)\cdot\lambda,T(x)\rangle \\&=\langle\lambda,T(x)\cdot\sigma(a)\rangle=
\langle\lambda,T(x\cdot a)\rangle\\&=\langle T^{\ast}(\lambda),x\cdot a\rangle=\langle a\cdot T^{\ast}(\lambda),x\rangle,
\end{align*}
so $ T^{\ast}(\sigma(a)\cdot\lambda)=a\cdot T^{\ast}(\lambda) $. Similarly $ T^{\ast}(\lambda\cdot\sigma(a))=T^{\ast}(\lambda)\cdot a $.
\end{proof}

The following characterization is useful.
\begin{theo}\label{2.5}
Let $ \mathcal{A} $ be a Banach algebra and let $ \sigma\in
Hom(\mathcal{A}) $. Then, the following are equivalent:
 \begin{enumerate}
  \item[(i)] $ \mathcal{A} $ is $ \sigma $-biflat;
  \item[(ii)] There is a $ \sigma $-$ \mathcal{A} $-bimodule homomorphism
  $ \rho:\mathcal{A}\rightarrow(\mathcal{A}\widehat{\otimes}\mathcal{A})^{\ast\ast} $ such that $ \pi^{\ast\ast}\circ\rho=\kappa_{\mathcal{A}}\circ\sigma
  $, where $\kappa_{\mathcal{A}}$ is the canonical embedding of $ \mathcal{A} $ into $\mathcal{A}^{\ast\ast}$.
 \end{enumerate}

\end{theo}
\begin{proof}
(i)$ \Longrightarrow $(ii) Since $ \mathcal{A} $ is $ \sigma $-biflat, there exists a bounded linear map
 $ \theta:(\mathcal{A}\widehat{\otimes}\mathcal{A})^{\ast}\rightarrow\mathcal{A}^{\ast} $ satisfying
 ($ 1 $) and $ \theta\circ\pi^{\ast}=\sigma^{\ast} $. We set $ \rho $ to be the restriction of $ \theta^{\ast} $ to $ \mathcal{A}
 $. Then, for every $ a\in\mathcal{A} $ and $ \lambda\in\mathcal{A}^{\ast} $ we
have
 \begin{align*}
 \langle\pi^{\ast\ast}\circ\rho(a),\lambda\rangle &=\langle\rho(a),\pi^{\ast}(\lambda)\rangle=\langle\theta^{\ast}(a),\pi^{\ast}(\lambda)\rangle \\
&=\langle a,\theta\circ\pi^{\ast}(\lambda)\rangle=\langle a,\sigma^{\ast}(\lambda)\rangle \\
&=\langle\sigma(a),\lambda\rangle=\langle\kappa_{\mathcal{A}}\circ\sigma(a),\lambda\rangle.
 \end{align*}
 Next, for every $ a,b\in\mathcal{A} $ and $ \xi\in(\mathcal{A}\otimes\mathcal{A})^{\ast} $
 \begin{align*}
 \langle \rho(ab),\xi\rangle&=\langle ab,\theta(\xi)\rangle=\langle b,\theta(\xi)\cdot a\rangle\\&=\langle b,\theta(\xi\cdot\sigma(a))\rangle=\langle\rho(b),\xi\cdot\sigma(a)\rangle=\langle\sigma(a)\cdot\rho(b),\xi\rangle,
 \end{align*}
 so $ \rho(ab)=\sigma(a)\cdot\rho(b) $. A Similar argument shows that $ \rho(ba)=\rho(b)\cdot\sigma(a) $.\\
(ii)$ \Longrightarrow $(i) Let $ \rho $ be as specified in the
clause (ii). Suppose that
 $ \tilde{\rho}:(\mathcal{A}\widehat{\otimes}\mathcal{A})^{\ast}\rightarrow\mathcal{A}^{\ast} $
 is the restriction of $ \rho^{\ast} $ into $
(\mathcal{A}\widehat{\otimes}\mathcal{A})^{\ast} $.
  Clearly $ \tilde{\rho} $ is a bounded linear map and satisfies  $ (1) $, by Lemma \ref{2.2}. For every $ a\in\mathcal{A} $ and $ \lambda\in\mathcal{A^{\ast}} $
 \begin{align*}
\langle\tilde{\rho}\circ\pi^{\ast}(\lambda),a\rangle &=\langle\pi^{\ast}(\lambda),\rho(a)\rangle=\langle\lambda,\pi^{\ast\ast}\circ\rho(a)\rangle \\
&=\langle\lambda,\kappa_{\mathcal{A}}\circ\sigma(a)\rangle=\langle\lambda,\sigma(a)\rangle=\langle\sigma^{\ast}(\lambda),a\rangle.
\end{align*}
showing that $ \tilde{\rho}\circ\pi^{\ast}=\sigma^{\ast} $.
\end{proof}

It is well-known that every biprojective Banach algebra is biflat.
The next proposition gives a generalization of this fact.
\begin{prop}\label{2.6}
Let $ \mathcal{A} $ be a Banach algebra and let $ \sigma\in Hom(\mathcal{A}) $. If $ \mathcal{A} $ is $ \sigma $-biprojective,
then $ \mathcal{A} $ is $ \sigma $-biflat.
\end{prop}
\begin{proof}
If $ \mathcal{A} $ is a $ \sigma $-biprojective Banach algebra, then there exists a bounded $ \sigma $-$ \mathcal{A} $-bimodule homomorphism
$ \rho:\mathcal{A}\rightarrow\mathcal{A}\widehat{\otimes}\mathcal{A} $ such that $ \pi\circ\rho=\sigma $.
For every $ \lambda\in\mathcal{A}^{\ast} $ and $ a\in\mathcal{A} $ we have
 \begin{equation*}
\langle\rho^{\ast}\circ\pi^{\ast}(\lambda),a\rangle=\langle\pi^{\ast}(\lambda),\rho(a)\rangle=\langle\lambda,\pi\circ\rho(a)\rangle=
\langle\lambda,\sigma(a)\rangle=\langle\sigma^{\ast}(\lambda),a\rangle.
 \end{equation*}
Since $ \rho $ is a $ \sigma $-$ \mathcal{A} $-bimodule
homomorphism, $ \rho^{\ast} $ satisfies $ (1) $ by Lemma \ref{2.2},
 so $ \mathcal{A} $ is $ \sigma $-biflat.
\end{proof}
The relation between biflatness and $ \sigma$-biflatness appears as
follows.
\begin{prop}\label{2.11}
Let $ \mathcal{A} $ be a Banach algebra and let $ \sigma\in
Hom(\mathcal{A}) $. Then:
\begin{enumerate}
\item[(i)] If $ \mathcal{A} $ is biflat, then $ \mathcal{A} $ is $ \sigma
$-biflat.
\item[(ii)] If $ \mathcal{A} $ is $ \sigma $-biflat and has a bounded
approximate identity, and if $ \sigma $ has a dense range, then $
\mathcal{A} $ is biflat.
\end{enumerate}
\end{prop}
\begin{proof}
(i) There exists a bounded $ \mathcal{A} $-bimodule homomorphism $
\rho:(\mathcal{A}\hat{\otimes}\mathcal{A})^{*}\longrightarrow\mathcal{A}^{\ast}
$ such that $ \rho\pi^{\ast}=i_{\mathcal{A}^{\ast}} $. Define $
\acute{\rho}=\sigma^{\ast}\rho:(\mathcal{A}\hat{\otimes}\mathcal{A})^{\ast}\longrightarrow\mathcal{A}^{\ast}
$. Then, for every $ \lambda\in
(\mathcal{A}\hat{\otimes}\mathcal{A})^{\ast} $ and $
a,b\in\mathcal{A} $, we have
\begin{align*}
\langle b,\acute{\rho}(\sigma(a)\cdot\lambda)\rangle&=\langle
b,\sigma^{\ast}\rho(\sigma(a)\cdot\lambda)\rangle=\langle
b,\sigma^{\ast}(\sigma(a)\cdot\rho(\lambda))\rangle\\&=\langle\sigma(b),\sigma(a)\cdot\rho(\lambda)\rangle=\langle\sigma(ba),\rho(\lambda)\rangle\\&=\langle
ba,\sigma^{\ast}\rho(\lambda)\rangle=\langle
b,a\cdot\acute{\rho}(\lambda)\rangle
\end{align*}
and analogously, $
\acute{\rho}(\lambda\cdot\sigma(a))=\acute{\rho}(\lambda)\cdot a $.
 It is clear that $ \acute{\rho}\pi^{\ast}=\sigma^{\ast} $, and hence $ \mathcal{A} $ is $ \sigma $-biflat.\\
(ii) By Proposition \ref{2.9} below, $ \mathcal{A} $ is amenable and
then by \cite[Theorem 2.9.65]{D}, $ \mathcal{A} $ is biflat.
\end{proof}
Now, we give a $ \sigma $-biflat Banach algebra which is not biflat.
\begin{example}
It is known that $ \mathcal{A}=l^{1}(\mathbb{Z}^{+}) $ is not
amenable, and so $ \mathcal{A}^{\sharp} $ is not amenable by
\cite[Corollary 2.3.11]{R}. Therefore, $ \mathcal{A}^{\sharp} $ is
not biflat by \cite{H}. We consider the homomorphism $
\sigma:\mathcal{A}^{\sharp}\rightarrow\mathcal{A}^{\sharp} $ defined
by $ \sigma(a+\lambda e)=\lambda $, $ (a\in \mathcal{A},\lambda\in
\mathbb{C}) $. An argument similar to \cite[Example 2.7]{G}, shows
that $ \mathcal{A}^{\sharp} $ is $ \sigma $-amenable. Then by Remark
\ref{2.10} below, $ \mathcal{A}^{\sharp} $ is $ \sigma $-biflat.
\end{example}

\section{Relation to  $ \sigma $-amenability}
Let $ \mathcal{A} $ be a Banach algebra, let $ \sigma \in
Hom(\mathcal{A})$, and let $X$ be a Banach $ \mathcal{A} $-bimodule.
A bounded linear map $ {D}:\mathcal{A}\rightarrow X $ is a
\textit{$\sigma$-derivation} if
$$ D(ab)=\sigma(a)\cdot D(b)+D(a)\cdot \sigma(b) \qquad  (a,b\in\mathcal{A}) ,
$$ and it is \textit{$\sigma$-inner
derivation} if there is an element $x\in X$ such that $
D(a)=\sigma(a)\cdot x-x\cdot \sigma(a) $ for all $ a\in\mathcal{A}
$.  A Banach algebra $ \mathcal{A} $ is \textit{$\sigma$-amenable}
if for every Banach $ \mathcal{A} $-bimodule $X$, every
$\sigma$-derivation $ {D}:\mathcal{A}\rightarrow X^* $ is
$\sigma$-inner \cite{M1,M2}. An element $
M\in(\mathcal{A}\widehat{\otimes}\mathcal{A})^{\ast\ast} $ is said
to be a \textit{$ \sigma $-virtual diagonal} for
 $ \mathcal{A} $, if $ \sigma(a)\cdot M=M\cdot\sigma(a) $ and $ \pi^{**}(\it M)\cdot\sigma(a)=\sigma(a) $ for all $ a\in\mathcal{A} $.
A bounded net $
(m_{\alpha})\subseteq\mathcal{A}\widehat{\otimes}\mathcal{A} $ is
said to be a \textit{$ \sigma $-approximate diagonal} for
 $ \mathcal{A} $ if $ \lim _{\alpha}m_{\alpha}\cdot\sigma(a)-\sigma(a)\cdot m_{\alpha}=0 $ and $ \lim_{\alpha}\pi(m_{\alpha})\cdot\sigma(a)=\sigma(a)
 $ for all $ a\in\mathcal{A} $ \cite {M}. In \cite[Proposition 2.4]{G} it is shown that if $ \mathcal{A} $ has
a $ \sigma $-virtual diagonal, then it has a $ \sigma $-approximate
diagonal. An easy verification shows that the converse is also true.

In the following two propositions, we establish a connection between
$ \sigma $-biflatness and existence of $ \sigma $-virtual diagonals.
\begin{prop}\label{2.7}
Let $ \mathcal{A} $ be a Banach algebra with a bounded approximate
identity and $ \sigma\in Hom(\mathcal{A}) $. If $ \mathcal{A} $ is $
\sigma $-biflat, then $ \mathcal{A} $ has a $ \sigma $-virtual
diagonal.
\end{prop}

\begin{proof}
Let $ (e_{\alpha}) $ be a bounded approximate identity for $ \mathcal{A} $. Since $ \mathcal{A} $ is $ \sigma $-biflat there exists a bounded
linear map $ \rho:(\mathcal{A}\widehat{\otimes}\mathcal{A})^{\ast}\rightarrow\mathcal{A}^{\ast} $ satisfying $ (1) $ such that
 $ \rho\circ\pi^{\ast}=\sigma^{\ast} $. We may suppose that $ \rho^{\ast}(e_{\alpha}) $ converges in the  weak*  topology to
 an element of $(\mathcal{A}\widehat{\otimes}\mathcal{A})^{\ast\ast}$, say $ M $. Then for each $ a\in\mathcal{A} $
 and $ \lambda\in(\mathcal{A}\widehat{\otimes}\mathcal{A})^{\ast} $, we have
 \begin{align*}
\langle \sigma(a)\cdot M,\lambda\rangle &=\langle
M,\lambda\cdot\sigma(a)\rangle= w^{\ast}-\lim _{\alpha}\langle
\rho^{\ast} (e_{\alpha}),\lambda\cdot\sigma(a)\rangle\\& =\lim
_{\alpha}\langle e_{\alpha},\rho(\lambda\cdot\sigma(a))\rangle=\lim
_{\alpha}\langle e_{\alpha},\rho(\lambda)\cdot a\rangle\\&=
\lim_{\alpha}\langle ae_{\alpha},\rho(\lambda)\rangle=\langle
a,\rho(\lambda)\rangle
 \end{align*}
 and similarly $ \langle M\cdot\sigma(a),\lambda\rangle=\langle a,\rho(\lambda)\rangle $. Thus $ \sigma(a)\cdot M=M\cdot\sigma(a) $.
An application of Theorem \ref{2.5} shows that
 \begin{align*}
\pi^{\ast\ast}_{\mathcal{A}}(M)\cdot\sigma(a)
&=w^{\ast}-\lim_{\alpha}\pi^{\ast\ast}_{\mathcal{A}}(\rho^{\ast}(e_{\alpha}))\cdot\sigma(a)=w^{\ast}-
\lim_{\alpha}\kappa_{\mathcal{A}}\circ\sigma(e_{\alpha})\cdot\sigma(a)=\sigma(a),
 \end{align*}
hence $ M $ is a $ \sigma $-virtual diagonal for $ \mathcal{A} $.
\end{proof}

\begin{prop}\label{2.8}
Let $ \mathcal{A} $ be a Banach algebra and  $ \sigma\in
Hom(\mathcal{A}) $. If $ \mathcal{A} $ has a $ \sigma $-virtual
diagonal, then $ \mathcal{A} $ is $ \sigma $-biflat.
\end{prop}

\begin{proof}
Let $ M\in(\mathcal{A}\hat{\otimes}\mathcal{A})^{\ast\ast} $ be a $ \sigma $-virtual diagonal. Define
 $ \rho:\mathcal{A}\rightarrow(\mathcal{A}\widehat{\otimes}\mathcal{A})^{\ast\ast} $ by $ \rho(a)=\sigma(a)\cdot M $ \
 $ (a\in\mathcal{A}) $. Clearly $ \rho $ is bounded, linear and $ \sigma $-$ \mathcal{A} $-bimodule homomorphism. Also
 \begin{equation*}
\pi^{\ast\ast}\circ\rho(a)=\pi^{\ast\ast}(\sigma(a)\cdot M)=\pi^{\ast\ast}(M)\cdot\sigma(a)=\sigma(a)=\kappa_{\mathcal{A}}(\sigma(a)).
 \end{equation*}
Thus $ \mathcal{A} $ is $ \sigma $-biflat.
\end{proof}

Now, we describe the relation between $ \sigma $-biflatness and $
\sigma $-amenability.
\begin{prop}\label{2.9}
Let $ \mathcal{A} $ be a $ \sigma $-biflat Banach algebra with a
bounded approximate identity. If $ \sigma\in Hom(\mathcal{A}) $ has
a dense range, then $ \mathcal{A} $ is amenable so is $ \sigma
$-amenable.
\end{prop}
\begin{proof}
Let $ \mathcal{A} $ be a $ \sigma $-biflat Banach algebra. By Proposition \ref{2.7} $ \mathcal{A} $ has a $ \sigma $-virtual
 diagonal, equivalently, $ \mathcal{A} $ has a $ \sigma $-approximate diagonal $ m_{\alpha}\subseteq(\mathcal{A}\hat{\otimes}\mathcal{A}) $.
  We show that $ \mathcal{A} $ has an approximate diagonal, so it is amenable by \cite[Theorem 2.9.65]{D} and thus by \cite[Corollay 2.2]{M}
   it is $ \sigma $-amenable, as required. Suppose that $ a\in\mathcal{A} $ and $ \varepsilon>0 $. There exists $ N $ such that for every
    $ n\geq N $, $ \lVert a-\sigma(a_{n})\rVert<\varepsilon $ then $ \lVert a-\sigma(a_{N})\rVert<\varepsilon $.
     There is $ \alpha_{0} $ such that for all $ \alpha\geq\alpha_{0} $, $ \lVert\sigma(a_{N})\cdot m_{\alpha}-m_{\alpha}\cdot\sigma(a_{N})\rVert<\varepsilon $.
     Hence for $ \alpha\geq\alpha_{0} $, we have $ \lVert a\cdot
m_{\alpha}-m_{\alpha}\cdot a\rVert\leq\lVert a\cdot
m_{\alpha}-\sigma(a_{N})\cdot m_{\alpha}\rVert+\lVert\sigma(a_{N})
\cdot m_{\alpha}-m_{\alpha}\cdot\sigma(a_{N})\rVert+\lVert
m_{\alpha}\cdot\sigma(a_{N})-m_{\alpha}\cdot
a\rVert\leq(2M+1)\varepsilon $,
 where $ M $ is a bounded of $ (m_{\alpha}) $. This shows that $ a\cdot m_{\alpha}-m_{\alpha}\cdot a\rightarrow 0 $ and
 similarly, $ a\pi(m_{\alpha})\rightarrow a $.
\end{proof}
\begin{rem}\label{2.10}
    There is a converse for Proposition \ref{2.9}. Indeed if $ \mathcal{A} $ is $ \sigma $-amenable with a bounded approximate identity,
    then it has $ \sigma $-virtual diagonal \cite[Theorem 2.2]{G}, and whence by Proposition \ref{2.8}, $ \mathcal{A} $ is $ \sigma $-biflat.
\end{rem}
We conclude the current section with two examples of $ \sigma
$-biflat Banach algebras which are not $ \sigma $-amenable. We refer
the reader to \cite[Definition 3.1.8 and Definition C.1.1]{R} for
the definitions of \textit{property} ($ \mathbb{A} $) and
\textit{approximation property} for Banach spaces. We also write $
\mathfrak{A}(E) $ for the space of \textit{approximable operators}
on a Banach space $E$.

  \begin{example}
    let $ E $ be a Banach space with property ($ \mathbb{A} $) such that $ E^{**} $ does not have
    the bounded approximation property. Then $ \mathfrak{A}(E^{*}) $ is biflat
    while it is not amenable \cite[Theorem 4.3.24]{R}. Therefore, $ \mathfrak{A}(E^{*}) $ is $ \sigma
    $-biflat for each $ \sigma\in Hom(\mathfrak{A}(E^{*})) $. On the
    other hand, One may check that every $ \sigma $-amenable Banach
    algebra for which  $ \sigma $ has a dense range, is amenable. Hence,
    choosing a homomorphism $ \sigma\in Hom(\mathfrak{A}(E^{*})) $ with a dense
    range, it is readily seen that $ \mathfrak{A}(E^{*}) $ is not $ \sigma $-amenable.
 \end{example}

\begin{example}
    let $ H $ be an infinite dimensional Hilbert space. It was shown
    in \cite[Example 4.3.25]{R} that $ \mathfrak{A}( H \widehat{\otimes}
    H)$ is biflat but not amenable. Then, an argument similar to Example
    3.1 shows that $ \mathfrak{A}( H \widehat{\otimes}
    H)$ is  $ \sigma
    $-biflat which is not $ \sigma $-amenable, whereas $ \sigma$ is a homomorphism in  $ Hom(\mathfrak{A}( H \widehat{\otimes}
    H)) $ with a dense range.
 \end{example}

\section{ The role of sequences }
We start with the following which is similar to Lemma 2.2.
\begin{lemm} Let $ \mathcal{A} $ be a Banach algebra, $X$ and $Y$ be Banach $ \mathcal{A} $-bimodule and let $ \sigma\in Hom(\mathcal{A})
$. If $T : X \longrightarrow Y$ is a bounded linear map satisfying
$T(\sigma(a)\cdot x) =  a\cdot Tx$ and $T(  x\cdot \sigma(a)) =
Tx\cdot a$ ($a \in \mathcal{A}, x \in X$), then $T^*$ is a
$\sigma$-$ \mathcal{A} $-bimodule homomorphism.
\end{lemm}

 For a Banach algebra $ \mathcal{A} $, we consider the short exact sequence
$$ \Sigma : 0 \longrightarrow ker \pi \stackrel \imath
\longrightarrow \mathcal{A} \widehat{\otimes}\mathcal{A} \stackrel
\pi \longrightarrow \mathcal{A} \longrightarrow 0 \ ,$$ and its
duals  $ \Sigma^*$ and $ \Sigma^{**}$, where $\imath$ is the natural
injection. It is known that a Banach algebra $ \mathcal{A}$ with a
bounded approximate identity is amenable if and only if $ \Sigma^*$
(and so $ \Sigma^{**}$) split \cite{C-L}.

\begin{prop} Let $ \mathcal{A} $ be a Banach algebra with a
bounded approximate identity and $ \sigma\in Hom(\mathcal{A}) $.
Suppose that $ \mathcal{A} $ is $ \sigma $-amenable. Then we have
the following statements:

$(i)$ Regarding the sequence $ \Sigma^*$, there is a bounded linear
map $ \theta : (\mathcal{A} \widehat{\otimes}\mathcal{A})^*
\longrightarrow \mathcal{A}^*$ such that $ \theta \pi^* = \sigma^*$,
and $$ \theta( \sigma(a)\cdot f ) = a\cdot \theta(f)  \ \ \
\text{and} \ \ \ \theta( f\cdot \sigma(a)) = \theta(f)\cdot a  \ \ \
( a \in \mathcal{A}, f \in (\mathcal{A}
\widehat{\otimes}\mathcal{A})^*) \ .
$$

$(ii)$ Regarding the sequence $ \Sigma^{**}$, there is a $ \sigma
$-$ \mathcal{A} $-bimodule homomorphism $ \gamma : \mathcal{A}^{**}
\longrightarrow (\mathcal{A} \widehat{\otimes}\mathcal{A})^{**}$
such that $ \pi^{**} \gamma = \sigma^{**}$.
\end{prop}
\begin{proof} $(i)$ Let $M \in (\mathcal{A}
\widehat{\otimes}\mathcal{A})^{**}$ be a $ \sigma $-virtual diagonal
for $ \mathcal{A} $. We define the map $ \theta : (\mathcal{A}
\widehat{\otimes}\mathcal{A})^* \longrightarrow \mathcal{A}^*$ via
$$ \langle a ,  \theta(f)  \rangle := \langle  f\cdot \sigma(a) , M
\rangle \ \ \ (a \in \mathcal{A}, f \in (\mathcal{A}
\widehat{\otimes}\mathcal{A})^*) \ .
$$ Then, for $a \in \mathcal{A}$ and $\lambda \in \mathcal{A}^*$
\begin{align*}
\langle a ,  \theta \pi^* \lambda \rangle &= \langle (\pi^*
\lambda)\cdot \sigma(a), M \rangle = \langle \pi^*( \lambda\cdot
\sigma(a)) , M \rangle = \langle  \lambda\cdot \sigma(a) , \pi^{**}
M \rangle
\\&= \langle  \lambda ,  \sigma(a) \pi^{**} M \rangle
= \langle  \lambda ,  \sigma(a)  \rangle = \langle  a , \sigma^*
\lambda \rangle
\end{align*}
so that $ \theta \pi^* = \sigma^*$. Next for $a , b \in \mathcal{A}$
and $f \in (\mathcal{A} \widehat{\otimes}\mathcal{A})^*$, we have
\begin{align*}
\langle b , \theta(\sigma(a) \ . \ f)  \rangle &= \langle (\sigma(a)
\ . \ f) \ . \  \sigma(b) , M \rangle = \langle  f \ . \  \sigma(b)
, M  \ . \ \sigma(a) \rangle \\&= \langle  f \ . \  \sigma(b) ,
\sigma(a) \ . \ M  \rangle = \langle  f \ . \  \sigma(ba) ,
 M  \rangle \\&= \langle ba , \theta(f) \rangle = \langle  b , a \ . \ \theta(f)
 \rangle
\end{align*}
whence $ \theta( \sigma(a) \ . \ f ) = a \ . \ \theta(f) $. The
equality $ \theta( f \ . \ \sigma(a)) = \theta(f) \ . \ a$ is even
easier.

$(ii)$ Take $ \gamma := \theta^*$, where $\theta$ is given by the
clause $(i)$. Then, it is immediate by Lemma 4.1.
\end{proof}

Finally, we generalize \cite[Proposition 4.3.23]{R} where the proof
reads somehow the same lines.
\begin{theo}
    Let $ \mathcal{A} $ be Banach algebra, let $ \mathcal{B} $ be a closed subalgebra of $ \mathcal{A} $ and
     $ \sigma\in Hom(\mathcal{A}) $ for which $ \sigma(\mathcal{B}) \subseteq \mathcal{B}
    $ with the followings properties:
    \begin{enumerate}
        \item[(i)] $ \mathcal{B} $ is $ \sigma $-amenable;
        \item[(ii) ] $ \mathcal{B} $ is a left ideal of $ \mathcal{A} $;
        \item[(iii)] $ \mathcal{B} $ has a bounded approximate identity which is also a bounded left approximate identity for $ \mathcal{A} $.
    \end{enumerate}
    Then $ \mathcal{A} $  is $ \sigma $-biflat.
 \end{theo}
 \begin{proof}
    Since $ \mathcal{B} $ is $ \sigma $-amenable, Proposition 4.2(ii) yields
    the existence of a $ \sigma $-$ \mathcal{B} $-bimodule homomorphism $
\rho:\mathcal{B}^{**}\longrightarrow(\mathcal{B}\hat{\otimes}\mathcal{B})^{**}
$ such that $ \pi^{**}_{\mathcal{B}} \circ
\rho=(\sigma\mid_{\mathcal{B}})^{**}$. Set  $
\tilde{\rho}=(\iota\hat{\otimes}\iota)^{**} \circ \rho $, where $
\iota:\mathcal{B}\hookrightarrow\mathcal{A} $ is the canonical
embedding. Let $ (e_{\alpha})_{\alpha} $ be a bounded approximate
identity for $ \mathcal{B} $ which is a
    left bounded approximate identity for $ \mathcal{A} $. For each
    $a \in \mathcal{A}$, $ (
    \tilde{\rho}(e_{\alpha})\cdot\sigma(a))_\alpha$ is a bounded net in $(\mathcal{A}\hat{\otimes}\mathcal{A})^{**}
    $, and so without loss of generality we may suppose that it is
    convergence. Therefore we obtain a bounded linear map $ \bar{\rho}:\mathcal{A}\longrightarrow(\mathcal{A}\hat{\otimes}\mathcal{A})^{**}
    $, defined by $ \bar{\rho}(a)=w^{*}-\lim_\alpha \tilde{\rho}(e_{\alpha})\cdot\sigma(a)
    $. It is immediate that $\bar{\rho}$ is a right $ \sigma
    $-$\mathcal{A}$-module homomorphism. To check that $\bar{\rho}$
    is a left $ \sigma
    $-$\mathcal{A}$-module homomorphism, we first notice that
    \begin{align*}
    \bar{\rho}(axb)&=w^{*}-\lim_\alpha\tilde{\rho}(e_{\alpha})\cdot\sigma(axb)\\
    &=w^{*}-\lim_\alpha\sigma(axb)\cdot\tilde{\rho}(e_{\alpha}) \ \ \ \   (\text{since} \  axb \in\mathcal{B}) \\
    &=w^{*}-\lim_\alpha \sigma(a) \sigma(xb)\cdot\tilde{\rho}(e_{\alpha}) \\
    &=w^{*}-\lim_\alpha\sigma(a)\cdot\tilde{\rho}(e_{\alpha})\cdot\sigma(xb) \ \ \ \   (\text{since} \  xb \in\mathcal{B}) \\
    &=\sigma(a)\cdot\bar{\rho}(x)\cdot\sigma(b) \ \ \  \ (a,x\in\mathcal{A}, \ b\in\mathcal{B}
    ).
    \end{align*}
    Let $a,x,b\in\mathcal{A}$. By Cohen's factorization theorem, $  b=cd $ for some $ c\in\mathcal{B}$ and $ d\in\mathcal{A}
    $. Then, it follows from the above observation that
    \begin{align*}
    \bar{\rho}(axb)&=\bar{\rho}(axcd)=\bar{\rho}(axc)\cdot\sigma(d)=\sigma(a)\cdot\bar{\rho}(x)\cdot\sigma(c)\cdot\sigma(d)\\
    &=\sigma(a)\cdot\bar{\rho}(x)\cdot\sigma(cd)=\sigma(a)\cdot\bar{\rho}(x)\cdot\sigma(b).
    \end{align*}
    Now, let $ a,x\in\mathcal{A} $. Again, by Cohen's  factorization theorem, there are $ y,z \in\mathcal{A} $ such that $ b=yz $.
    Then
    $$
    \bar{\rho}(ax)=\bar{\rho}(ayz)=\sigma(a)\cdot\bar{\rho}(y)\cdot\sigma(z)=\sigma(a)\cdot\bar{\rho}(yz)=\sigma(a)\cdot\bar{\rho}(x).
    $$
    So $ \bar{\rho} $ is $ \sigma $-$ \mathcal{A} $-bimodule homomorphism.
    It remains to prove that $ \pi^{**}_{\mathcal{A}}\circ\bar{\rho}=\kappa_{\mathcal{A}}\circ\sigma $. For each $ a\in\mathcal{A} $ we have
    \begin{align*}
    \pi^{**}_{\mathcal{A}} (\bar{\rho}(a))&=w^{*}-\lim_\alpha \pi^{**}_{\mathcal{A}}(\tilde{\rho}(e_{\alpha})\cdot\sigma(a))
    =w^{*}-\lim_\alpha \pi^{**}_{\mathcal{A}}(\tilde{\rho}(e_{\alpha}))
    \ . \ \sigma(a) \\
    &=w^{*}-\lim_\alpha \pi^{**}_{\mathcal{B}}(\rho(e_{\alpha}))
    \ . \ \sigma(a) =\lim_\alpha  \sigma^{**}(e_{\alpha}) \ . \ \sigma(a) =
    \lim_\alpha \kappa_{\mathcal{A}} ( \sigma(e_{\alpha} a)) \\&=
    \kappa_{\mathcal{A}} (\sigma(a)) .
    \end{align*}
    So by Theorem \ref{2.5}, $ \mathcal{A} $ is $ \sigma $-biflat.
 \end{proof}

\section{ Application for Tensor products}

Let $ \mathcal{A} $ be a Banach algebra, let $ \sigma\in Hom(\mathcal{A}) $ and let $ C>0 $. We say that $ \mathcal{A} $ is $ C $-$ \sigma $-biflat,
 If there exists a map $ \rho:(\mathcal{A}\otimes\mathcal{A})^{\ast}\longrightarrow\mathcal{A}^{\ast} $, satisfying $ (1) $
 such that $ \rho\circ\pi^{\ast}=\sigma^{\ast} $ and $ \lVert \rho\rVert< C $.

 Let $ \mathcal{A} $ and $ \mathcal{B} $ be Banach algebras and
  $ \sigma_{\mathcal{A}}\in Hom(\mathcal{A}) $ and $ \sigma_{\mathcal{B}}\in Hom(\mathcal{B}) $.
   Let $ E $ be a Banach $ \mathcal{A} $-bimodule, and let $ F $ be a  Banach $ \mathcal{B} $-bimodule.
   We regard $ E\widehat{\otimes}F $ as a Banach $ \mathcal{A}\widehat{\otimes}\mathcal{B} $-bimodule with the actions
 \begin{equation*}
  \begin{array}{rl}
(a\otimes b)\cdot(x\otimes y)=(a\cdot x)\otimes(b\cdot y)\\
(x\otimes y)\cdot(a\otimes b)=(x\cdot a)\otimes(y\cdot b),
  \end{array}
 \end{equation*}
for every $ a\in\mathcal{A},b\in\mathcal{B},x\in E $ and $ y\in F $.

Using Ramsden's notation \cite{P}, we construct $
\tilde{\mathcal{B}}(E,F) $ and $ \hat{\mathcal{B}}(F,E) $ as
follows.

 Let $ \tilde{\mathcal{B}}(E,F) $ be the space $ \mathcal{B}(E,F) $ with the following module actions:
  $$
   \begin{array}{rl}
((a\otimes b)\cdot T)(x)=\sigma_{\mathcal{B}}(b)\cdot T(x\cdot\sigma_{\mathcal{A}}(a))\\
(T\cdot(a\otimes b))(x)=T(\sigma_{\mathcal{A}}(a)\cdot x)\cdot\sigma_{\mathcal{B}}(b),
    \end{array}
  $$
for every $ T\in\mathcal{B}(E,F)$, $a\in\mathcal{A}$, $b\in\mathcal{B}$, $x\in E$. Consider the map
 $ \tilde{T}: (E\widehat{\otimes}F)^{\ast}\rightarrow{\tilde{\mathcal{B}}}(E,F^{*});
   \lambda\rightarrow\tilde{T}(\lambda) $ defined by $ \tilde{T}(\lambda)(x)(y)=\lambda(x\otimes y) $.
   By \cite[42. Proposition 13]{B}, $ \tilde{T} $ is an isometric isomorphism of Banach spaces.
   We claim that $ \tilde{T} $ satisfies (1). Indeed
\begin{align*}
\tilde{T}(\sigma_{\mathcal{A}\otimes\mathcal{B}}(a\otimes b)\cdot\lambda)(x)(y)&=\tilde{T}((\sigma_{\mathcal{A}}(a)\otimes\sigma_{\mathcal{B}}(b))
\cdot\lambda)(x)(y)\\&=((\sigma_{\mathcal{A}}(a)\otimes\sigma_{\mathcal{B}}(b))\cdot\lambda)(x\otimes y)\\&=\lambda((x\otimes y)
\cdot(\sigma_{\mathcal{A}}(a)\otimes\sigma_{\mathcal{B}}(b)))\\&=\lambda(x\cdot\sigma_{\mathcal{A}}(a)\otimes y\cdot\sigma_{\mathcal{B}}(b)),
\end{align*}
 and on the other hand
\begin{align*}
(a\otimes b)\cdot\tilde{T}(\lambda)(x)(y)&=\sigma_{\mathcal{B}}(b)\cdot\tilde{T}(\lambda)(x\cdot
\sigma_{\mathcal{A}}(a))(y)\\&=\sigma_{\mathcal{B}}(b)\cdot\lambda(x\cdot\sigma_{\mathcal{A}}(a)\otimes y)\\&=
\lambda(x\cdot\sigma_{\mathcal{A}}(a)\otimes y)\cdot\sigma_{B}(b)\\&=\lambda(x\cdot\sigma_{\mathcal{A}}(a)\otimes y\cdot\sigma_{\mathcal{B}}(b)).
\end{align*}

 Let $ \widehat{\mathcal{B}}(F,E) $ be $ \mathcal{B}(F,E) $ with the following module actions
\begin{equation*}
 \begin{array}{rl}
((a\otimes b)\cdot T)(y)=\sigma_{\mathcal{A}}(a)\cdot T(y\cdot\sigma_{\mathcal{B}}(b))\\
(T\cdot(a\otimes b))(y)=T(\sigma_{\mathcal{B}}(b)\cdot y)\cdot\sigma_{\mathcal{A}}(a),
 \end{array}
\end{equation*}
 for all $ T\in\mathcal{B}(F,E)$, $a\in A$, $b\in B$, $y\in F$. Consider the map $ \widehat{T}
 :(E\widehat{\otimes}F)^{\ast}\rightarrow{\hat{\mathcal{B}}}(F,E^{\ast});\lambda\mapsto\widehat{T}(\lambda) $ defined by $ \widehat{T}(\lambda)(y)(x)=\lambda(x\otimes y) $. A similar argument as we use for $ \tilde{T} $, shows that $ \hat{T} $ is an isometric isomorphism of Banach spaces satisfying (1)

 Now, we are ready to prove the main goal of the current section.

\begin{theo}
 Let $ \mathcal{A} $ be $ C_{1} $-$ \sigma_{\mathcal{A}} $-biflat Banach algebra and let $ \mathcal{B} $ be
  $ C_{2} $-$ \sigma_{\mathcal{B}} $-biflat Banach algebra and  if both $ \sigma_{\mathcal{A}} $ and $ \sigma_{\mathcal{B}} $
  are idempotents. Then $ (\mathcal{A}\widehat{\otimes}\mathcal{B}) $ is $ C_{1}  C_{2} $-$ \sigma_{\mathcal{A}}\otimes\sigma_{\mathcal{B}} $-biflat.
\end{theo}

\begin{proof}
 There exists a bounded linear map $ \rho_{\mathcal{A}}:(\mathcal{A}\widehat{\otimes}\mathcal{A})^{\ast}\rightarrow\mathcal{A}^{\ast} $
 such that $ \rho_{\mathcal{A}}\circ\pi^{\ast}_{\mathcal{A}}=\sigma^{\ast}_{\mathcal{A}} $, satisfying (1)
  and $ \Vert\rho_{\mathcal{A}}\Vert\leq C_{1} $ and also a bounded linear map
   $ \rho_{\mathcal{B}}:(\mathcal{B}\widehat{\otimes}\mathcal{B})^{\ast}\rightarrow\mathcal{B}^{\ast} $
   such that $ \rho_{\mathcal{B}}\circ\pi^{\ast}_{\mathcal{B}}=\sigma^{\ast}_{\mathcal{B}} $, satisfying (1)
    and $ \Vert\rho_{\mathcal{B}}\Vert\leq C_{2} $. Consider the composition
 \begin{align*}
 \rho_{\mathcal{A}\widehat{\otimes}\mathcal{B}}:
 ((\mathcal{A}\widehat{\otimes}\mathcal{B})\widehat{\otimes}(\mathcal{A}\widehat{\otimes}\mathcal{B}))
 ^{\ast}&\cong\tilde{\mathcal{B}}(\mathcal{A}\widehat{\otimes}\mathcal{A},(\mathcal{B}\widehat{\otimes}\mathcal{B})^{\ast})\\
 &\xrightarrow{T\mapsto\rho_\mathcal{B}\circ T}\tilde{\mathcal{B}}(\mathcal{A} \hat{ \otimes} \mathcal{A}, \mathcal{B}^\ast))\cong\widehat{\mathcal{B}}
 (\mathcal{B},(\mathcal{A}\widehat{\otimes}\mathcal{A})^{\ast})\\
 &\xrightarrow{T\mapsto\rho_\mathcal{A}\circ T}\widehat{\mathcal{B}}(\mathcal{B},\mathcal{A}^{*})\cong(\mathcal{A}\widehat{\otimes}\mathcal{B})^{\ast}.
 \end{align*}
All the maps satisfy the equation (1), and $ \Vert\rho_{\mathcal{A}\widehat{\otimes}\mathcal{B}}\Vert\leq C_{1}C_{2} $.
For $ \lambda\in(\mathcal{A}\widehat{\otimes}\mathcal{B})^{\ast} $, we follow
$ \pi^{\ast}_{\mathcal{A}\widehat{\otimes}\mathcal{B}}(\lambda) $ under the sequence of composition
$ \rho_{\mathcal{A}\widehat{\otimes}\mathcal{B}} $. We have
$$
  \rho_{\mathcal{B}}\circ\pi^{\ast}_{\mathcal{A}\widehat{\otimes}\mathcal{B}}(\lambda)\mapsto \rho_{\mathcal{B}}\circ\pi^{\ast}_{\mathcal{B}}\circ\tilde{T}
  _{\lambda}\circ\pi_{\mathcal{A}}\mapsto\sigma^{\ast}_{\mathcal{B}}\circ\tilde{T}_{\lambda}\circ\pi_{\mathcal{A}}\mapsto\pi^{\ast}_{\mathcal{A}}\circ\widehat{T}
  _{\lambda}\circ\sigma_{\mathcal{B}}.
 $$
So
\begin{align*}
\rho_{\mathcal{A}}\circ\rho_{\mathcal{B}}\circ\pi^{\ast}_{\mathcal{A}\widehat{\otimes}\mathcal{B}}(\lambda)\mapsto\rho_{\mathcal{A}}
\circ\pi^{\ast}_{\mathcal{A}}\circ\widehat{T}_{\lambda}\circ\sigma_{\mathcal{B}}
\mapsto\sigma^{\ast}_{\mathcal{A}}\circ\widehat{T}_{\lambda}\circ\sigma_{\mathcal{B}}
\mapsto\sigma^{\ast}_{\mathcal{A}\widehat{\otimes}\mathcal{B}}(\lambda),
\end{align*}
hence $\rho_{\mathcal{A}\widehat{\otimes}\mathcal{B}}\circ\pi^{\ast}_{\mathcal{A}\widehat{\otimes}\mathcal{B}}=
\sigma^{\ast}_{\mathcal{A}\widehat{\otimes}\mathcal{B}} $.

Note that the last equality comes from the following isomorphisms
$$
(\mathcal{A}\widehat{\otimes}\mathcal{B})^{\ast}\cong\widetilde{\mathcal{B}}(\mathcal{A},\mathcal{B}^{\ast})
\cong\widehat{\mathcal{B}}(\mathcal{B},\mathcal{A}^{\ast}).
$$
 \end{proof}



\begin{thebibliography}{99}
\setlength{\baselineskip}{.45cm}
\bibitem{B}
 F. F. Bonsall and J. Duncan, Complete Normed Algebras, Springer-Varlag, Berlin Heidelberg New York, 1973.
\bibitem{C-L}
P. C. Curtis, R. J. Loy, The structure of amenable Banach algebras,
\textit{J. London Math. Soc.}, (2) \textbf{40} (1989), 89-104.

\bibitem{D}
H. G. Dales, Banach Algebras And Automatic continuity, London
Mathematical Society Monographs 24, Clarendon Press, Oxford, 2000.

\bibitem{G}
 Z. Ghorbani and M. L. Bami, $\varphi$-amenable and  $\varphi$-biflat Banach algebras, {\it Bull. Iranian Math. Soc.} \textbf{39} (3) (2013), 507-515.


\bibitem{H}
 A. Y. Helemskii, The Homology of Banach and Topological Algebras, Dordrecht, Netherlands, Kluwer, 1989.

\bibitem{J}
 B. E. Johnson, Cohomology in Banach algebras, \textit{Mem. Amer. Math. Soc.} \textbf{127} (1972).

\bibitem{J1}
 B. E. Johnson, Approximate diagonals and Cohomology of certain annihilator Banach algebras, {\it Amer. J. Math.} \textbf{94} (1972), 685-698.

\bibitem{M1}
 M. Mirzavaziri and M. S. Moslehian, $\sigma$-derivation in Banach algebras, {\it Bull. Iranian Math. Soc.} \textbf{32} (1) (2006), 65-78.

\bibitem{M2}
 M. Mirzavaziri and M. S. Moslehian, $\sigma$-amenability of Banach algebras, {\it Southeast Asian Bull. Math.} \textbf{33} (2009), 89-99.

\bibitem{M}
 M. S. Moslehian and A. N. Motlagh, Some notes on ($\sigma,\tau$)-amenable of Banach algebras, {\it Stud. Univ. Babe\c s-Bolyai Math.}
   \textbf{53} (3) (2008), 57-68.

\bibitem{P}
 P. Ramsden, Biflatness of semigroup algebras, \textit{Semigroup Forum}, \textbf{79} (2009), 515-530.

\bibitem{R}
 V. Runde, Lectures on Amenability, Lecture Notes in Mathematics 1774, Springer-Verlag, Berlin, 2002.


\bibitem{Y}
 T. Yazdanpanah and H. Najafi, $\sigma$-contractible and  $\sigma$-biprojective Banach algebras, {\it Quaestiones Math.} \textbf{33} (2010), 485-495.














\end{thebibliography}
\end{document}